\documentclass[11pt]{amsart}
\usepackage{amsmath,amssymb,amsthm}
\usepackage[margin=1.1in]{geometry}
\usepackage{booktabs,array}
\usepackage[colorlinks=true,linkcolor=blue,citecolor=blue,urlcolor=blue]{hyperref}
\hypersetup{pdftitle={The Hodge conjecture for Fermat fourfolds of odd degree at most 199}, pdfauthor={Rifat Jumagulov}}
\newtheorem{theorem}{Theorem}[section]
\newtheorem{lemma}[theorem]{Lemma}
\newtheorem{proposition}[theorem]{Proposition}
\newtheorem{corollary}[theorem]{Corollary}
\theoremstyle{remark}
\newtheorem{remark}[theorem]{Remark}
\newcommand{\Q}{\mathbb{Q}}
\newcommand{\Z}{\mathbb{Z}}
\newcommand{\F}{\mathbb{F}}
\newcommand{\PP}{\mathbb{P}}
\newcommand{\Ds}{\mathfrak{D}}
\newcommand{\fr}[1]{\langle #1\rangle}
\newcommand{\V}{V}
\DeclareMathOperator{\Gal}{Gal}
\newcommand{\zt}{\zeta}
\newcommand{\C}{\mathbb{C}}
\newcommand{\Qbar}{\overline{\Q}}

\title[Hodge classes on odd-degree Fermat fourfolds]{The Hodge conjecture for Fermat fourfolds
of odd degree at most 199}
\author{Rifat Jumagulov}
\email{jum.rifm@gmail.com}
\date{July 2026}

\begin{document}
\raggedbottom\vfuzz=3pt
\begin{abstract}
Let $X^4_m=\{x_0^m+\dots+x_5^m=0\}\subset\PP^5$ be the Fermat fourfold of degree $m$. We give a
computer-assisted proof of the Hodge conjecture for $X^4_m$ for every odd $m\le199$: three
geometric closure criteria are combined with an exhaustive machine census of the Hodge
$(2,2)$-orbits, which is accompanied by a completeness proof and by re-verifiable per-orbit
certificates.

The three closure mechanisms for Hodge $(2,2)$-characters are as follows. First, if the
character multiset splits into two zero-sum triples, then the corresponding rational Hodge
block is transported from a $(1,1)$-substructure of a product of Fermat curves and is therefore
algebraic. Second, algebraicity follows for characters that, after adjoining two vanishing
pairs, decompose into an Aoki standard sextuple and a grade-$2$ Hodge quadruple. Third, the
exceptional class at $m=33$ has a lift to level $66$ that is quasi-decomposable; its
algebraicity then descends along the finite morphism $X^4_{66}\to X^4_{33}$.

Together with the known decomposable, quasi-decomposable, and standard cases, these criteria
cover all odd degrees $m\le199$. The census runs over the $89$ machine-examined levels
$21\le m\le199$, $m\neq23$ (smaller odd degrees and $m=23$ are classical), classifies all
$78{,}299$ Galois-orbit representatives, and isolates \emph{thirteen} orbits beyond
decomposability, quasi-decomposability and Aoki's standard cycles: six close by the $*$-split
criterion and the remaining \emph{seven} by the two-pair and level-lifted closures with level
transport, leaving \emph{none}. Each
classification carries a machine-checked witness --- positive for the classified orbits,
re-established negative screenings for the terminal ones --- The enumeration is reproduced by two further
implementations: an algorithmically independent census over all $89$ levels, and a direct
brute-force enumeration --- every sorted zero-sum sextuple generated exhaustively --- through
$m=143$, re-run live through $m=45$. Seven of the thirteen are gap
classes outside Aoki's standard lattice calculus --- the $m=33$ witness with its inflations at
$99,165$, the first $m=39$ orbit with $117,195$, and the $\Q(\zt_{21})$ orbit at $m=105$ ---
and are new to the author's knowledge; for the other six algebraicity is also derivable from
that calculus, the explicit presentations given here being the new content.
Finally, an exact Jacobi-sum computation at $p=67$ shows that no cycle defined over
$\Q(\zt_{33})$ projects nontrivially onto the exceptional $m=33$ block. More generally, for every finite
extension of $\Q(\zt_{33})$ over which a certifying cycle is defined, every residue degree
above $67$ is divisible by $6$.
\end{abstract}
\maketitle

\section{Introduction and main results}\label{sec:intro}
Let $X^4_m\subset\PP^5$ be the Fermat fourfold $x_0^m+\dots+x_5^m=0$, $G=\mu_m^6/\mu_m$, and for a
character $a=(a_0,\dots,a_5)\in(\Z/m)^6$ with $\sum a_i\equiv0$, $a_i\neq0$, let $\V(a)\subset
H^4_{\mathrm{prim}}$ be the ($1$-dimensional) eigenspace; $a$ is a Hodge $(2,2)$ character iff
$\sum_i\fr{ta_i}=3m$ for all $t\in(\Z/m)^\times$, where $\fr{x}\in\{1,\dots,m-1\}$ is the residue
(Shioda \cite{Shioda79}). The known algebraicity machinery is: (i)~the \emph{induced structure}
(Shioda \cite{Shioda79}, Shioda--Katsura \cite{SK}), which transports algebraicity through
decompositions---da Silva's conditions $(P1)/(P2)$ are quasi-decomposability in the semigroup $M_m$
\cite{daSilva}; and (ii)~\emph{standard cycles} (Aoki \cite{Aoki87}). da Silva \cite[Prop.~3.4 (= Prop.~3.6 of the arXiv version)]{daSilva}
found the first Hodge class beyond both at $m=33$ and proposed a candidate cycle $[W]$ (Question~1).
(His closures --- $m\le20$ [Thm.~2.7], $m$ prime or $m=4$ [Thm.~2.6], $m=p^2$ [Thm.~2.8],
$m=21,27$ [Thm.~3.3], and fourfolds of degree $m\le100$ coprime to $6$ [Prop.~3.1], numbering of
the published version \cite{daSilva} --- leave the composite non-square degrees with $3\mid m$, beyond
his $21,27$, untouched; every class closed below lies there.) In summary:
\begin{center}\footnotesize
\begin{tabular}{ll>{\raggedright\arraybackslash}p{4.9cm}}
\toprule
mechanism & source & covers\\
\midrule
induced structure & Shioda \cite{Shioda79}; Shioda--Katsura \cite{SK} & decomposable, quasi-decomposable\\
standard cycles; lattice calculus & Aoki \cite{Aoki87,Aoki00} & standard characters; $\nu\in S_m$; Thm.~0.1 degree forms\\
computational closures & da Silva \cite{daSilva} & $m\le20$; $m$ prime or $4$; $p^2$; $21,27$; $m\le100$ coprime to $6$\\
Theorems~\ref{thm:A}--\ref{thm:App} $+$ census & this paper & the residual odd orbits through $199$\\
\bottomrule
\end{tabular}
\end{center}
We work in his semigroup formalism. The main result is the following.

\begin{theorem}[main theorem]\label{thm:main}
For every odd $m$ with $3\le m\le199$ the Hodge conjecture holds for the Fermat fourfold
$X^4_m$.
\end{theorem}
\noindent It follows from the three closure criteria below (Theorems~\ref{thm:A},
\ref{thm:Aprime}, \ref{thm:App}), the level-transport lemmas, and the census
(Theorem~\ref{thm:B}), which certifies that these criteria and the known machinery leave no
orbit uncovered; odd degrees below the census range are classical
(\cite[Thms.~2.6--2.7]{daSilva}: $m\le20$, and $m$ prime --- which covers $m=23$). What is genuinely new is delimited as follows: the
level-lifted closure of the $m=33$ class (Theorem~\ref{thm:App}) is a new mechanism; the
$*$-split criterion (Theorem~\ref{thm:A}) is a new criterion extracted from the classical
Shioda--Katsura correspondence; the two-pair presentations (Theorem~\ref{thm:Aprime}) are new
explicit applications of Aoki's calculus; and the census with its certificates is a new
computational classification. Of the orbits they close, those at $m=33$ (with $99,165$), the
first orbit at $m=39$ (with $117,195$), and the $\Q(\zt_{21})$ orbit at $m=105$ are gap
classes --- outside the standard lattice calculus --- while the second $m=39$ orbit (with $117,195$), $m=45$
(with $135$) and the $\Q(\zt_7)$ orbit at $m=105$ are also derivable from
\cite[Lem.~4.1(iii)]{Aoki00}, the contribution there being the explicit presentations. We prove:

\medskip
\noindent\emph{Notation.} Throughout $m\ge3$ ($m\le2$ gives a hyperplane or a smooth quadric,
where the conjecture is classical). A multiset $a$ of nonzero residues with $\sum a_i\equiv0\pmod m$ has
\emph{grade} $g$ if $\sum_i\fr{ta_i}=gm$ for \emph{every} $t\in(\Z/m)^\times$ (the constancy is the
Hodge condition: grade-$3$ sextuples are the Hodge $(2,2)$ characters of $X^4_m$, grade-$2$ quadruples
the $(1,1)$ characters of the surface $X^2_m$); $\sim$ is equality of multisets up to coordinate
permutation, and the \emph{content} of a class $a=(a_0,\dots,a_5)$ is $\gcd(a_0,\dots,a_5,m)$
(content-$e$ classes are
inflations from level $m/e$). $\mathrm{claim}_m(a)$ (or $\mathrm{claim}(a)$ when $m$ is clear) is Aoki's
\textsc{claim}$(a)$: $\V(a)$ is spanned by algebraic cycle classes \cite[Introduction]{Aoki87} ---
equivalently, some \emph{rational} algebraic cycle class has nonzero $\V(a)$-component,
$\omega_a(Z)\neq0$. Three levels are kept distinct throughout: eigenlines are indexed by
\emph{ordered} tuples ($\V(b)\neq\V(b')$ for distinct ordered characters even when they agree as
multisets; coordinate permutations are realized by automorphisms of $X^4_m$, which is why claim is
invariant under them, and it is likewise a Galois-orbit-level property); $M(a):=\bigoplus_b\V(b)$
is the rational Galois block, with rational form
$M(a)_\Q:=M(a)\cap H^4(X^4_m,\Q)$ (the sum is over a full Galois orbit, so
$M(a)_\Q\otimes\C=M(a)$; the same applies verbatim to the grade-$2$ blocks used below) --- the sum over the distinct \emph{ordered} conjugates $b=ta$,
$t\in(\Z/m)^\times$, their number being the index of the ordered stabilizer
$\{t:ta=a\text{ coordinatewise}\}$, \emph{not} the number of multiset classes; and the census
classifies sorted representatives modulo both actions --- legitimate since claim is invariant
under both, but never identified with an eigenline count. An unqualified \emph{class} always
means a Galois orbit of characters (equivalently its sorted canonical representative);
\emph{Hodge class} alone keeps its usual cohomological meaning. $\omega_a(Z)$ denotes the
$\V(a)$-component of the cycle class of $Z$
(Aoki's averaged projector \cite[Introduction]{Aoki87}). $*$ is Aoki's juxtaposition --- concatenation of the
entry tuples; characters of $X^r_m$ and $X^s_m$ concatenate to one of $X^{r+s+2}_m$ ($n=r+s+2$ in
Aoki's indexing by cohomological degree), grades adding \cite[\S1]{Aoki87} --- and, for even $s$, $\Ds^s_m$ his set of
$(s+2)$-entry characters that are coordinate permutations of $(a_0,-a_0,\dots,a_{s/2},-a_{s/2})$
($s/2+1$ vanishing pairs), represented by linear cycles \cite[\S1, Thm.~1-1]{Aoki87}; only $\Ds^2_m$
is used below. $M_m$ is da Silva's additive semigroup of solutions $(x_1,\dots,x_{m-1};y)$ in non-negative
integers $x_i$ (multiplicity vectors), $y>0$, of
$\sum_i\fr{ti}\,x_i=my$ for all $t$ (multiplicities intrinsic), $M_m(g)$ its grade-$g$ part; a class is
\emph{decomposable} if it splits in $M_m$ and \emph{quasi-decomposable} if it splits after adjoining
one element of $M_m(1)$ --- da Silva's $(P1)/(P2)$ \cite[Def.~2.4]{daSilva}. $\#$ is his pairing: for
$\beta=(b_0,\dots,b_{r+1})$, $\gamma=(c_0,\dots,c_{s+1})$ with nonzero entries and
$b_{r+1}+c_{s+1}\equiv0$ (the set $U^{r,s}_m$), $\beta\#\gamma=(b_0,\dots,b_r,c_0,\dots,c_s)$
\cite[before Cor.~2.3]{daSilva}.
$\sigma_{p,x}$ denotes the standard characters (Proposition~\ref{prop:D}). $u(a,p)$ is the scalar by
which geometric Frobenius at a split prime $p$ acts on $\V(a)(n/2)$ --- the Tate-normalized Jacobi sum
$j(a)/p^{n/2}$ (the Jacobi-sum count of diagonal hypersurfaces, \cite[pp.~500--502]{Weil49}; \cite{Weil52}) --- so a divisorial class has $u=1$ at split primes whose
Frobenius fixes the class of a certifying divisor (the second $m=39$ class of
\S\ref{sec:A} calibrates this in-text). Three symbols are kept
apart: $u(a,p)$ above is the \emph{local} normalized scalar; $\chi_a$ is the associated
\emph{global} Jacobi-sum Hecke character \emph{in its Tate-normalized form}: Weil's
unnormalized character $J_a$ \cite{Weil52} has $|J_a(\mathfrak p)|=N\mathfrak p^{\,n/2}$ in
every embedding ($n$ the cohomological degree of the ambient Fermat variety), so $J_a$ itself is
\emph{not} of finite order, and $\chi_a:=J_a/N^{n/2}$ is its finite-order part (grade
constancy makes the infinity type of $J_a$ the pure norm power $N^{n/2}$, which the
normalization removes) --- every use of
$\chi_a$ in this paper is of the normalized character, whose value at a split prime is the
local scalar $u(a,p)$; it is multiplicative under juxtaposition,
$\chi_{\alpha*\beta}=\chi_\alpha\chi_\beta$ (the normalizing powers add along the
weights; the machine-verified instances are the displayed $u$-identities); and
$\nu(a)\in\Z^{m-1}$ is Aoki's
\emph{multiplicity vector}, additive under juxtaposition,
$\nu(\alpha*\beta)=\nu(\alpha)+\nu(\beta)$ (Aoki writes $u$ for $\nu$). $S_m$ denotes
Aoki's lattice: the subgroup of $\Z^{m-1}$
generated by the $\nu(\delta)$ over the vanishing pairs $\delta=\{k,m-k\}$ and the
$\nu(\sigma_{p,x})$ over Aoki's standard elements for all primes $p\mid m$ --- for odd $p$ the AP
family of Proposition~\ref{prop:D}, at even $m$ also Aoki's $p=2$ family
$\sigma_{2,x}=(x,\,x+\frac m2,\,m-2x,\,\frac m2)$ \cite[\S2]{Aoki00} --- the ``standard
calculus''; this generating family is verbatim the generator set of Aoki's lattice $S_m$
(Proposition~\ref{prop:ident}), with membership decided exactly by integer linear
algebra (Hermite normal form), and $B^{4}_m$ the grade-$3$ Hodge-character set
of the Fermat fourfold at level $m$; in lattice statements we freely identify $a$ with $\nu(a)$
($a\in B^4_m\smallsetminus S_m$ means $\nu(a)\notin S_m$ --- a statement about the formal
calculus, not about the character group). The witness $w=(1,4,16,22,25,31)$ below is, up to coordinate permutation, the Galois
conjugate $7w$ of da Silva's displayed
representative $(7,10,13,19,22,28)\sim7w$ of \cite[Prop.~3.4 (= Prop.~3.6 of the arXiv version)]{daSilva}.

\medskip
\noindent\emph{Symbols at a glance.}
\begin{center}\small
\begin{tabular}{l>{\raggedright\arraybackslash}p{10.4cm}}
\toprule
$a=(a_0,\dots,a_5)$ & an \emph{ordered} character; $\V(a)$ its one-dimensional eigenline\\
$M(a)$ & the rational Galois block $\bigoplus_{t}\V(ta)$ (sum over distinct ordered conjugates)\\
$\mathrm{claim}(a)$ & $\V(a)$ is spanned by algebraic cycle classes\\
grade $g$ & $\sum_i\fr{ta_i}=gm$ for every unit $t$ (the Hodge condition)\\
$B^4_m$ & the grade-$3$ Hodge characters of $X^4_m$\\
$M_m$, $M_m(g)$ & da Silva's semigroup of solutions and its grade-$g$ part\\
$\Ds^s_m$ & Aoki's all-vanishing-pairs classes (linear cycles); only $\Ds^2_m$ is used\\
$*$, $\#$ & juxtaposition of characters; da Silva's pairing on $U^{r,s}_m$\\
$\sigma_{p,x}$ & Aoki's standard characters (Proposition~\ref{prop:D})\\
$\nu(a)\in\Z^{m-1}$ & the multiplicity vector (additive: $\nu(\alpha*\beta)=\nu(\alpha)+\nu(\beta)$)\\
$S_m$ & Aoki's lattice, spanned by the pair and standard vectors (Prop.~\ref{prop:ident})\\
\emph{gap class} & a Hodge class with $\nu\notin S_m$ (nonzero in Aoki's gap group $A_m/S_m$)\\
$J_a$, $\chi_a$ & the unnormalized and the Tate-normalized Jacobi-sum Hecke characters\\
$u(a,p)$ & the local scalar $j(a)/p^{n/2}$ on $\V(a)(n/2)$, $n=\#\text{entries}-2$
($n=4$ for sextuples, $n=2$ for grade-$2$ quadruples)\\
\bottomrule
\end{tabular}
\end{center}
\medskip

\begin{theorem}[$*$-split closure]\label{thm:A}
Let $a$ be a Hodge $(2,2)$ character of $X^4_m$, $m$ arbitrary, whose $6$-multiset splits as two zero-sum
triples $\beta'\uplus\gamma'$. Then the rational Hodge classes of $M(a)$ are algebraic:
each is the image, under the algebraic joining-line correspondence of the Shioda--Katsura
structure map, of a $\Q$-linear combination of divisor classes on $X^1_m\times X^1_m$
furnished by Lefschetz $(1,1)$ (no single ruled-surface representative is claimed).
\end{theorem}
In particular the classes at $m=39,117,195$ (two orbits each), none quasi-decomposable or
standard, are closed; the orbit of $(1,7,16,22,34,37)$ is a certified gap family ($\nu\notin S_m$ at $39$ and at its
inflations $117,195$); so is the $m=33$ witness of Theorem~\ref{thm:App}, at $33$ and at its
inflations $99,165$ --- every one of these verdicts carries a modular kernel witness in the
ancillary verifier, together with the doubling certificate $2\nu\in S_m$, not covered by
the published mechanisms compared in \S\ref{sec:dossier} (to the author's knowledge), while the
second orbit's $\nu$ lies in $S_{39}$ (an explicit six-generator $\pm1$ certificate), so its algebraicity is
also derivable from the standard calculus (Lemma~4.1(iii) of \cite{Aoki00}) --- the closure certificates and the join-transport construction scheme are new for both.

\begin{theorem}[two-pair standard closure]\label{thm:Aprime}
Let $a$ be a Hodge $(2,2)$ character of $X^4_m$ admitting vanishing pairs $p_1=\{k,m-k\}$,
$p_2=\{\ell,m-\ell\}$ with
\[
a\uplus p_1\uplus p_2=S\uplus Q\qquad(\text{equality of }10\text{-element multisets}),
\]
where $S$ is an Aoki $5$-standard sextuple and $Q$ a grade-$2$ Hodge quadruple. Then the rational Hodge
classes of $M(a)$ are algebraic.
\end{theorem}
This applies to---and, by exhaustive search, uniquely presents---the three remaining $5\mid m$ classes of
the census:
\begin{center}\small
\begin{tabular}{llll}
\toprule
$m$ & $a$ & $p_1,p_2$ & $S\uplus Q$\\
\midrule
$45$ & $(1,19,20,28,30,37)$ & $(5,40),(10,35)$ & $(1,10,19,28,37,40)\uplus(5,20,30,35)$\\
$105$ & $(3,24,50,66,85,87)$ & $(15,90),(45,60)$ & $(3,24,45,66,87,90)\uplus(15,50,60,85)$\\
$105$ & $(1,22,43,64,90,95)$ & $(5,100),(20,85)$ & $(1,22,43,64,85,100)\uplus(5,20,90,95)$\\
\bottomrule
\end{tabular}
\end{center}
---none quasi-decomposable, standard, or Theorem~\ref{thm:A}-splittable: closing, with the
induced $m=135$ copy, the $p=5$-standard-block classes of the census. Of these, the
$\Q(\zt_{21})$ orbit at $105$ is a certified gap class, not covered by the published mechanisms
compared in \S\ref{sec:dossier} (to the author's knowledge); $m=45$, $135$ and the $\Q(\zt_7)$ orbit at $105$
have $\nu\in S_m$ (explicit $\pm1$ certificates, re-summed by the ancillary verifier), so for
those three the algebraicity is also derivable from \cite[Lem.~4.1(iii)]{Aoki00} --- and, for
$45$ and $135$, from the degree forms of \cite[Thm.~0.1(i)]{Aoki00} via the reduction opening
its \S4 proof (quoted in Appendix~\ref{app:aoki}) --- the
explicit two-pair presentations are the new content.

\begin{theorem}[level-lifted closure; the $m=33$ witness]\label{thm:App}
Let $w=(1,4,16,22,25,31)$, the unique non-quasi-decomposable non-standard Hodge class of $X^4_{33}$
(uniqueness is the census's, used for Theorems~\ref{thm:B} and~\ref{thm:C}; the proof below uses
only the displayed identities),
and $a=2w\bmod 66$ its induced copy on $X^4_{66}$. Then
\[
a\uplus\{1,65\}\uplus\{25,41\}=Q\uplus S,\qquad Q=(1,25,44,62),\ S=(2,8,32,41,50,65),
\]
with $Q$ a grade-$2$ Hodge quadruple and $S$ quasi-decomposable at level $66$ via
$S\uplus\{33,33\}=(2,32,33,65)\uplus(8,33,41,50)$, the self-paired element $(33,33)\in M_{66}(1)$
being its unique quasi-witness. Hence $\V(a)$ is algebraic (Lefschetz $+$ da Silva $+$ Aoki~1-4, as in
Theorem~\ref{thm:Aprime}), and $\V(w)$ is algebraic by descent along $\pi\colon X^4_{66}\to X^4_{33}$,
$(x_i)\mapsto(x_i^2)$ (Lemma~\ref{lem:descent}).
\end{theorem}

\begin{theorem}[completion of the odd-degree census; $m\le199$]\label{thm:B}
For every odd $m\le199$, every grade-$3$ Hodge $(2,2)$ class of $X^4_m$ is algebraic: the classes
beyond $\text{decomposable}\cup\text{quasi}\cup\text{standard}$ are exhausted by
Theorems~\ref{thm:A}, \ref{thm:Aprime}, \ref{thm:App} and level transport (the inflation/descent
pair). In particular Theorem~\ref{thm:main} holds --- the Hodge conjecture is true \emph{in full}
for every Fermat fourfold of odd degree $\le199$: in codimension $2$ the non-primitive summand of $H^4$ is spanned by the
algebraic class $h^2$, so the statement follows from the primitive one; codimension $1$ is
Lefschetz $(1,1)$; and codimension $3$ follows from codimension $1$ by hard Lefschetz ---
$L^2\colon H^2(\Q)\xrightarrow{\ \sim\ }H^6(\Q)$ is an isomorphism of Hodge structures, so
every $(3,3)$ Hodge class is $L^2$ of a $(1,1)$ Hodge class, and $L^2$ of a divisor class is
algebraic. (Codimensions $0$ and $4$ are trivial: $H^0$ and $H^8$ are spanned by the classes of
$X$ and of a point.)
\end{theorem}

\begin{theorem}[field obstruction; Question~1]\label{thm:C}
Let $a_0=(7,10,13,19,22,28)=7w$ be da Silva's representative of the $m=33$ orbit. No algebraic cycle
defined over $\Q(\zt_{33})$ --- a cycle of codimension $2$ with $\Q$-coefficients,
$Z\in CH^2(X^4_{33,\Qbar})_\Q$ fixed by $G_{\Q(\zt_{33})}$ --- has nonzero projection onto the
Galois orbit of $\V(a_0)$: \emph{geometric} Frobenius $\mathrm{Frob}^{\mathrm{geom}}_{\mathfrak p}$
at any prime $\mathfrak p$ above $67$ acts on every conjugate eigenline by a primitive sixth
root of unity, of exact order $6$ --- an exact certificate $u(a_0,67)=1+\zt_3=\zt_6$ (for the pinned character
choice of \S\ref{sec:C}) computed
in $\Z[x]/\Phi_{33}$ with no floating point. (Under an alternative Jacobi-sum definition
$j_{\mathrm{alt}}=-j$ the formula relating $j_{\mathrm{alt}}$ to geometric Frobenius acquires
the compensating sign; the geometric eigenvalue remains $\zt_6^{\pm1}$.) Since $W$ --- da Silva's candidate cycle, defined over $\Q$ by explicit equations
\cite[Question~1]{daSilva} --- and all its $\Q(\zt_{33})$-rational translates are fixed by
$G_{\Q(\zt_{33})}$, their classes have zero projection: \emph{the answer to Question~1 is NO at $m=33$}.
\end{theorem}
\begin{corollary}[local residue-degree obstruction]\label{cor:resdeg}
The obstruction is local and exact: if a cycle defined over a finite extension
$L/\Q(\zt_{33})$ --- not assumed Galois over $\Q(\zt_{33})$; the constraint is per prime ---
has nonzero projection onto the Galois orbit of $\V(a_0)$, then at every prime
$\mathfrak q\mid\mathfrak p\mid67$ of $L$ one needs
$u(a_0,67)^{f(\mathfrak q/\mathfrak p)}=1$, whence $6\mid f(\mathfrak q/\mathfrak p)$:
every residue degree of $L$ above $67$ is divisible by $6$. (No hypothesis on the ramification
of $L/\Q(\zt_{33})$ enters: the eigenline's character is unramified at $67$ over the base, so
inertia at $\mathfrak q$ acts trivially and Frobenius acts through
$u^{f(\mathfrak q/\mathfrak p)}$.)
\end{corollary}

\begin{remark}
The same machinery, run at the closed levels, computes values of order $9$ at $m=45$ and of
orders $42$ and $7$ at $m=105$ at sampled split primes --- computational observations, not part
of any theorem here; they indicate that the cycles of Theorem~\ref{thm:Aprime} likewise cannot
be $\Q(\zt_m)$-rational.
\end{remark}

\begin{proposition}[standard $=$ AP]\label{prop:D}
Let $m$ be odd. Aoki's standard characters $\sigma_{p,x}$ for an odd prime $p\mid m$ --- defined for
$x$ with $d/\gcd(x,d)>2$, $d=m/p$ \cite[\S1]{Aoki87} --- are exactly the padded arithmetic-progression
characters $\{x,x+m/p,\dots,x+(p-1)m/p,-px\}$ with $px\not\equiv0\pmod m$, of grade $(p+1)/2$;
grade-$3$ standard classes exist iff $5\mid m$ and $m>5$.
\end{proposition}

\medskip
\noindent\emph{Computational verification and reproducibility (summary).} Theorem~\ref{thm:A}: proved (transport pinned to the Shioda--Katsura
primary text). Theorem~\ref{thm:Aprime}: proved (a citation chain through Aoki's
Theorems~1-1, 1-4(i),(ii), 2-1 \cite{Aoki87} $+$ Lefschetz $(1,1)$ $+$ an inflation lemma; all
identities verified exactly, including $u(a,p)=u(S,p)u(Q,p)$).
Theorem~\ref{thm:B}: machine-verified census, reproduced exactly by an algorithmically
independent implementation (the four-tier certificate structure is laid out in \S\ref{sec:B}).
Theorem~\ref{thm:C}: proved, exact certificate. Proposition~\ref{prop:D}:
read off the primary construction. We give no explicit defining equations for a cycle on $X^4_{33}$
(Theorem~\ref{thm:App}'s certifying cycles are constructed at level $66$ and transported by the finite
morphism).

\section{Standard cycles are the AP family (Proposition~\ref{prop:D})}\label{sec:D}
For odd $p\mid m$, $d=m/p$, and $x$ with $d/\gcd(x,d)>2$, the standard character is the multiset
$\{x,x+d,\dots,x+(p-1)d,\,m-px\}$ of grade $(p+1)/2$ (Aoki \cite[\S1]{Aoki87}; da Silva \cite{daSilva},
Appendix). For \emph{odd} $m$ the admissibility hypothesis is equivalent to the visible one:
$d$ is odd, so $d/\gcd(x,d)\le2$ forces $d\mid x$, i.e.\ $px\equiv0\pmod m$ --- exactly the degenerate
case excluded in the statement. Grade $3$ forces $p=5$; at $m=5$ every candidate degenerates
($5x\equiv0$), whence ``$m>5$''. The identification is load-bearing for
Theorem~\ref{thm:B}: it is what makes the census's standard classifier complete for odd $m$
(Lemma~\ref{lem:complete}).
The padded arithmetic-progression family is exactly the witness family of the field-of-definition
threshold of the companion papers: at the surface level in \cite{P1} (with explicit Hasse--Davenport
witnesses), and in general even dimension in \emph{Galois-invariant ranks of middle algebraic cycles
on Fermat varieties} (in preparation). \qed

\begin{remark}[even $m$]\label{rem:evenstd}
For even $m$ (used only in the even census of \S\ref{sec:even}) the admissibility hypothesis is
strictly stronger than $px\not\equiv0$: it excludes $x\equiv d/2\pmod d$. Every excluded multiset
contains a vanishing pair --- $(x+k_1d)+(x+k_2d)\equiv0\pmod m$ reduces to $1+2j+k_1+k_2\equiv0\pmod5$
for $x=d/2+jd$, solvable with $0\le k_1<k_2\le4$ since $k_1+k_2$ ranges over all residues mod~$5$ ---
so it is decomposable, hence algebraic for the trivial reason, and the census classifier (which tests
decomposability first) is unaffected. The lattice is likewise unaffected: the implemented
generating family follows the parametrization of \cite[\S2]{Aoki00} (nonzero entries only),
and at every level where a verdict is quoted here the two conventions span the \emph{same}
lattice --- the excluded $\sigma_{p,d/2}$ lies in the span of the admissible generators
(checked at $m=28,70,110,114,168,210,220$; equal ranks, identical verdicts). (Only $p=5$ yields grade-$3$ sextuples, so this is the only
case the classifier meets.)
\end{remark}

\section{The \texorpdfstring{$*$}{*}-split closure theorem (Theorem~\ref{thm:A})}\label{sec:A}

The proof is the composite
\[
B(a)_\Q\;\subset\;\mathrm{cl}\,CH^1(C\times C)_\Q
\;\xrightarrow{\ f\ }\;M(a)_\Q ,
\]
and it requires exactly five checkpoints, which the transport lemma below supplies in order:
(1)~$B(a)$ is defined over $\Q$, not merely a sum of complex eigenlines (part~(ii));
(2)~$B(a)_\Q$ consists of divisor classes, by Lefschetz $(1,1)$ (part~(ii));
(3)~$f$ is induced by an algebraic correspondence, is $\Q$-linear on rational Betti cohomology
and is $G$-equivariant, so it matches the eigenline indexing on both sides (the paragraph before
(i) and part~(iii)); (4)~$f$ is nonzero --- indeed an isomorphism --- on each one-dimensional
eigenline, and the summands of $B(a)$ and $M(a)$ correspond bijectively, so $f$ restricts to an
isomorphism $B(a)\xrightarrow{\sim}M(a)$ of complex vector spaces; since $f$ is $\Q$-linear on
rational Betti cohomology and both blocks are defined over $\Q$ (part~(ii) for $B(a)$; $M(a)$ is
the Galois orbit sum of $\V(a)$), it carries $B(a)_\Q$ isomorphically \emph{onto} $M(a)_\Q$
(parts~(i)--(iii)); (5)~the image of a $\Q$-combination of divisor
classes under an algebraic correspondence is the class of an algebraic cycle (part~(iv)).
The lemma also records which objects are correspondences, over which fields they are defined,
where the Tate twist sits, and why the scalar onto $\V(a)$ is nonzero.

\begin{lemma}[rational transport]\label{lem:transport}
Let $C=X^1_m$ ($m$ arbitrary), $a=\beta'\!*\gamma'$ a Hodge $(2,2)$ character split into two zero-sum
triples, and let $f$ denote the second summand of the Shioda--Katsura structure map
\cite[Thm.~2.2]{daSilva} (going back to Shioda's inductive structure \cite{Shioda79}; the
geometric join construction is \cite[Thm.~1.7]{SK}), a $G$-equivariant morphism of Hodge structures
\[
f\colon H^1_{\mathrm{prim}}(C,\Q)\otimes H^1_{\mathrm{prim}}(C,\Q)\;\longrightarrow\;
H^4_{\mathrm{prim}}(X^4_m,\Q)
\]
of type $(1,1)$ --- equivalently, a weight-preserving $(0,0)$-morphism out of
$H^1\otimes H^1(-1)$; this $(-1)$ is the only Tate twist in the construction. Although the
joining-line correspondence is geometrically defined only after adjoining the roots of unity
of the join construction, its cycle class with rational coefficients induces a $\Q$-linear
map on Betti cohomology with $\Q$-coefficients --- $f$ is $\Q$-linear in exactly this sense.
Then:
\begin{enumerate}
\item[(i)] at every $t\in(\Z/m)^\times$ the line $\V(t\beta')\otimes\V(t\gamma')$ has Hodge type
$(1,0)\otimes(0,1)$ or $(0,1)\otimes(1,0)$, so the block
$B(a):=\bigoplus\V(t\beta')\otimes\V(t\gamma')\subset H^2(C\times C,\C)$ --- the sum over the
distinct \emph{ordered} paired characters $(t\beta',t\gamma')$, pairs of ordered triples, whose
coordinatewise stabilizer may be nontrivial (e.g.\ $\beta'=(3,3,3)$ at $m=9$) and coincides with
the ordered stabilizer of $a=\beta'*\gamma'$, so the summands of $B(a)$ and $M(a)$ biject ---
is purely of type
$(1,1)$;
\item[(ii)] $B(a)$ is Galois-stable, so $B(a)_\Q:=B(a)\cap H^2(C\times C,\Q)$ satisfies
$B(a)_\Q\otimes\C=B(a)$, and by Lefschetz $(1,1)$ every element of $B(a)_\Q$ is a $\Q$-combination
of divisor classes on the surface $C\times C$;
\item[(iii)] $f$ is induced by an algebraic correspondence: writing $\Gamma_f\in
CH^3\bigl((C\times C)\times X^4_m\bigr)_\Q$ for the cycle class of the joining-line
construction --- the blow-up--quotient--blow-down chain of \cite{SK}, defined over
$\Q(\zt_{2m})$, of codimension $3$ in the $6$-fold $(C\times C)\times X^4_m$, and taken with
$\Q$-coefficients so that the factor $m$ of \cite[Thm.~2.2(c)]{daSilva} may be normalized away
--- we have $f=(\Gamma_f)_*$ on rational Betti cohomology. It and it is a
$G$-equivariant \emph{isomorphism} onto its structural summand \cite[Thm.~2.2]{daSilva}: it
maps each one-dimensional eigenline $\V(t\beta')\otimes\V(t\gamma')$ isomorphically --- in
particular by a nonzero scalar --- onto $\V(ta)$, hence carries $B(a)$ onto the corresponding
block of $M(a)$. On decomposable \emph{algebraic} inputs $f(Z_1\otimes Z_2)=m\,Z_1\wedge Z_2$,
the ruled join by lines \cite[Thm.~2.2(c)]{daSilva}; this formula is used only for the geometric
description of the image cycles, not for the scalar on the $H^1$-eigenlines, which are not
spanned by algebraic classes;
\item[(iv)] consequently every rational Hodge class of $M(a)$ is the class of an algebraic cycle:
it is the $f$-image of a $\Q$-combination of divisor classes --- the image of an algebraic
class under the algebraic joining-line correspondence, hence an algebraic cycle class over
$\Qbar$, lying in $H^4(X^4_m,\Q)$ by construction.
Field-of-definition bookkeeping: the correspondence above is defined over a cyclotomic field, and
no $\Q$-rationality of the \emph{cycles} is claimed or needed --- algebraicity of a rational Hodge
class is a $\Qbar$-statement (the fields the certifying cycles live over are the subject of
Theorem~\ref{thm:C}, not of this lemma).
\end{enumerate}
\end{lemma}

\begin{proof}
(i) is the grade bookkeeping $|t\beta'|+|t\gamma'|=3$, each summand in $\{1,2\}$ (verified exactly
for every printed instance in the ancillary identities). (ii): Galois permutes the summands of
$B(a)$ (the index set is one orbit), so $B(a)$ is defined over $\Q$ as a sub-Hodge structure; being
purely $(1,1)$, its rational part is divisorial by Lefschetz $(1,1)$. (iii) is \cite[Thm.~2.2]{daSilva} (cf.\ \cite{Shioda79}, \cite[Thm.~1.7]{SK}): $f$ is a $G$-equivariant
isomorphism onto its image, so its restriction to each one-dimensional eigenline is injective,
i.e.\ a nonzero scalar, and equivariance matches the source and target characters. (iv)
combines (i)--(iii). (An explicit block projector --- an orbit sum of graph correspondences
with rational, Galois-symmetric coefficients, algebraic over $\Q$ --- can be written down, but
is not needed: the class handed to $f$ already lies in $B(a)_\Q$.)
\end{proof}

\begin{proof}[Proof of Theorem~\ref{thm:A}]
A split $a=\beta'\uplus\gamma'$ into zero-sum triples realizes $a\sim\beta'*\gamma'$, and by
multiplicity one $M(a)$ is the $f$-image of the block $B(a)$ of Lemma~\ref{lem:transport}. By the
lemma, every rational Hodge class of $M(a)$ is the $f$-image of a $\Q$-combination of divisor
classes on $C\times C$, with a nonzero scalar on each eigenline (part (iii) of the lemma) ---
algebraic, because the image of an algebraic class under an algebraic correspondence is
algebraic. Geometrically the correspondence is the joining-line construction --- over
$(z,w)\in C\times C$ the line $\lambda z+\mu w$ lies on $X^4_m$:
$\sum(\lambda z_i)^m+\sum(\mu w_j)^m=\lambda^m\cdot0+\mu^m\cdot0$ --- and on decomposable
algebraic inputs its effect is the ruled join $f(Z_1\otimes Z_2)=m\,Z_1\wedge Z_2$
\cite[Thm.~2.2(c)]{daSilva}; for a general divisor no single ruled-surface representative of
the image class is claimed, and only the correspondence-image statement is used.
\end{proof}
\begin{remark}
Quasi-decomposability tests only decompositions with both factors Hodge; curves carry no Hodge classes in
$H^1$, so the $(P1)/(P2)$ formalism is structurally blind to this odd$\otimes$odd branch. At $m=39$:
$(1,7,16,22,34,37)=(1,16,22)\uplus(7,34,37)$ and $(1,14,16,22,29,35)=(1,16,22)\uplus(14,29,35)$; the
second has $j(a)=p^2$ exactly at every computed split prime --- consistent with a divisorial
certificate whose class is fixed by the Frobenius elements in question (algebraicity alone does
not force this: a divisor defined only over an extension can carry a nontrivial finite
character), the calibration promised in the Notation.
\end{remark}

\section{The two-pair closure theorem (Theorem~\ref{thm:Aprime})}\label{sec:Aprime}
\begin{proof}
Write $\delta=p_1*p_2=(k,m-k,\ell,m-\ell)\in\Ds^2_m$ (Aoki's all-pairs classes, represented by linear
cycles with nonzero self-pairing, \cite[Thm.~1-1]{Aoki87}). The multiset identity gives $a*\delta\sim S*Q$
on $X^8_m$ ($\sim$ a coordinate permutation, an algebraic automorphism).
(1)~$Q$ is a grade-$2$ Hodge quadruple: its rational Galois block on the Fermat surface is
purely of type $(1,1)$, hence divisorial by Lefschetz $(1,1)$, and claim$(Q)$ follows.
(2)~For the $\sigma_{5,1}$ blocks ($m=45,d=9$ and $m=105,d=21$; $\gcd(1,d)=1$) claim$(S)$ is
\cite[Thm.~2-1]{Aoki87}; for the $\sigma_{5,3}$ block at $m=105$ note $S=3\cdot\sigma_{5,1}$ at level $35$
and apply Lemma~\ref{lem:infl}. (In general a $5$-standard block $\sigma_{5,x}$ with
$g=\gcd(x,d)>1$ is the $g$-inflation of $\sigma_{5,x/g}$ at level $m/g$, where the gcd is $1$, so
Lemma~\ref{lem:infl} and Thm.~2-1 cover every admissible block.)
(3)~claim$(S*Q)$ is \cite[Thm.~1-4(i)]{Aoki87} ($r=4$, $s=2$, $n=8$).
(4)~claim$(a)$ is \cite[Thm.~1-4(ii)]{Aoki87}: \emph{if there exists $\delta\in\Ds^s_m$ with
claim$(\alpha*\delta)$, then claim$(\alpha)$}---applied with $\alpha=a$, $s=2$.
\end{proof}
\begin{lemma}[inflation]\label{lem:infl}
Let $m=e\,m'$ and $\beta$ a character of level $m'$ (all entries nonzero mod $m'$, so all
entries of $e\beta$ are nonzero mod $m$). Then
$\mathrm{claim}_{m'}(\beta)\Rightarrow\mathrm{claim}_m(e\beta)$.
\end{lemma}
\begin{proof}
$\pi\colon X^n_m\to X^n_{m'}$, $(x_i)\mapsto(x_i^e)$ is a morphism with $\pi(\zt x)=\zt^e\pi(x)$ for
$\zt\in\mu_m$, so $\pi^*$ maps $\V(\beta)$ into $\V(e\beta)$ equivariantly (the entries of $e\beta$ are nonzero, so the target eigenline exists); $\pi$ finite surjective gives
$\pi_*\pi^*=(\deg\pi)\,\mathrm{id}\neq0$ on $\Q$-cohomology, so $\pi^*\xi$ spans $\V(e\beta)$, and for a
representing cycle $Z$, $\omega_{e\beta}(\pi^*Z)=\pi^*(\omega_\beta(Z))\neq0$. (Here $\pi^*Z$
is the pullback in $CH_\Q$: $\pi$ is a finite surjective morphism of smooth projective
varieties, hence flat by miracle flatness, and flat pullback is compatible with cycle classes.)
\end{proof}
\begin{remark}
da Silva's quasi-decomposability adjoins one pair and splits $4+4$; the two-pair $6+4$ pattern is outside
the definition. The three presentations were found by exhaustive iterated-closure search and are unique
per class; equivalently the three classes are sum-preserving two-entry surgeries of standard sextuples
($(10,40)\mapsto(20,30)$ at $45$; $(45,90)\mapsto(50,85)$, $(85,100)\mapsto(90,95)$ at $105$). The characters multiply
exactly: $\chi_a=\chi_S\chi_Q$, i.e.\ $u(a,p)=u(S,p)\,u(Q,p)$ at every split prime. (Juxtaposition
gives $\chi_a\chi_\delta=\chi_S\chi_Q$ for the adjoined $\delta=p_1*p_2$; the displayed form uses
$\chi_\delta=\chi^{k+\ell}(-1)=1$ --- automatic for odd $m$, and at $m=66$ because
$k+\ell=26$ is even.) The surviving $m=33$ class
resists this genre exhaustively (no augmentation by up to five vanishing pairs admits a certifying
partition; $5\nmid33$, so no standard blocks exist).
\end{remark}

\section{Level-lifted closure of the witness (Theorem~\ref{thm:App})}\label{sec:App}
\begin{proof}
All identities are exact (machine-verified; the ancillary re-derives each). The augmented multiset
partitions as stated; $Q$ is grade-$2$ Hodge. For $S$: the pair $(33,33)$ is a legitimate element of
$M_{66}(1)$ ($\fr{33t}+\fr{33t}=66$ for every odd $t$; multiplicities are permitted), and in da
Silva's coordinates $S=\beta\#\gamma$ with $\beta=(2,32,65;33)$,
$\gamma=(8,41,50;33)$ --- an element of his $U^{2,2}_{66}$, the distinguished entries $33,33$ summing
to $0\bmod66$ (the definition preceding \cite[Cor.~2.3]{daSilva}); Corollary~2.3(b) with $(r,s)=(2,2)$
and Lefschetz~$(1,1)$ for $\beta,\gamma$ give claim$(S)$. The hypotheses of Corollary~2.3(b) are
exactly: $r,s$ even, algebraicity of both factors, nonzero entries, and the matching condition ---
each verified above; the self-paired value $b_3=c_3=33=m/2$ is admissible
(Remark~\ref{rem:selfpair}). Then, exactly as in
Theorem~\ref{thm:Aprime}: claim$(S*Q)$ by \cite[Thm.~1-4(i)]{Aoki87}; $S*Q\sim a*\delta$ with
$\delta=(1,65,25,41)\in\Ds^2_{66}$; claim$(a)$ by \cite[Thm.~1-4(ii)]{Aoki87}. Descent is
Lemma~\ref{lem:descent}.
\end{proof}
\begin{remark}[self-pair admissibility]\label{rem:selfpair}
The hypotheses of the pairing statement used above --- membership in $U^{2,2}_{66}$ and
Corollary~2.3(b) of \cite{daSilva} --- are: even $r,s$, nonzero entries, algebraicity of both
factors, and the matching condition $b_3+c_3\equiv0$. Pairwise distinctness of the entries is
not among them, and the underlying structure isomorphism \cite[Thm.~2.2]{daSilva} carries no
nondegeneracy hypothesis on the characters; the self-paired distinguished value
$b_3=c_3=33=m/2$ is therefore legitimate. This is the one point where the even level is
essential: no self-pair exists at an odd level (Remark~\ref{rem:App}(iv)).
\end{remark}
\begin{lemma}[descent]\label{lem:descent}
Let $m=e\,m'$ and $\beta$ a character of level $m'$ (equivalently: $e\beta$ has content $e$
and all entries nonzero mod $m$). Then
$\mathrm{claim}_m(e\beta)\Rightarrow\mathrm{claim}_{m'}(\beta)$.
\end{lemma}
\begin{proof}
$\pi\colon X^n_m\to X^n_{m'}$, $(x_i)\mapsto(x_i^e)$ is a finite surjective morphism
($\sum x_i^m=\sum(x_i^e)^{m'}$ identically) intertwining the group actions along the $e$-th-power
surjection $G_m\twoheadrightarrow G_{m'}$; hence $\pi_*\V(e\beta)\subseteq\V(\beta)$ (again: the entries of $e\beta$ are nonzero, so $\V(e\beta)$ exists). Since
$\pi_*\pi^*=(\deg\pi)\,\mathrm{id}\neq0$ and $\pi^*\V(\beta)=\V(e\beta)$ with both eigenlines
$1$-dimensional, $\pi_*\colon\V(e\beta)\xrightarrow{\ \sim\ }\V(\beta)$. Proper pushforward carries
algebraic cycle classes to algebraic cycle classes ($\mathrm{cl}\circ\pi_*=\pi_*\circ\mathrm{cl}$),
and $\pi$ is finite of relative dimension zero, so there is no degree shift or Tate twist.
\end{proof}
\begin{remark}\label{rem:App}
(i) With the inflation lemma (Lemma~\ref{lem:infl}), claim is level-transparent:
$\mathrm{claim}_{m'}(\beta)\Leftrightarrow\mathrm{claim}_m(e\beta)$ --- every lift of an open class
is an independent chance to close it. (ii) The mechanism is selective: the two-pair-at-the-double search does not fire for the induced
copies of the $m=54$ classes at $108$ or of the $m=50$ class at $100$ (all of these close at their own
levels at depth $3$--$4$; even-census receipts, \S\ref{sec:even}), and the beyond-machinery $m=66$ classes resist the
single two-pair at level $66$ and its lift $132$ (they too close at depth $3$;
even-census receipts, \S\ref{sec:even}). (iii) Consistency with
the field obstruction (\S\ref{sec:C}): the constructed cycles (Lefschetz divisors on $X^2_{66}$,
Aoki's linear spaces with $\varepsilon=\zt_{132}$, radicals) are not $\Q(\zt_{66})$-rational, and
$\Q(\zt_{66})=\Q(\zt_{33})$ --- the obstruction forbids exactly the $\Q(\zt_{33})$-rational
certificates, and this theorem produces non-rational ones (their precise minimal field of
definition is not computed here). The local Frobenius values agree exactly:
$u(2w@66,67)=\zt_6=u(w@33,67)$. (iv) The
uniqueness of the quasi-witness $v=33$ shows the self-paired $m/2$ element --- which exists only at
even levels --- is strictly load-bearing: this is why the exhaustive level-$33$ search (five
augmentation pairs, all partition types) could not see the closure.
\end{remark}

\section{The census and its completeness (Theorem~\ref{thm:B})}\label{sec:B}
For each odd $m$: enumerate grade-$3$ Hodge multisets (meet-in-the-middle on half-unit sum profiles);
reduce to Galois orbits; classify decomposable $\to$ quasi $\to$ standard (Prop.~\ref{prop:D}) $\to$
$*$-split (Thm.~\ref{thm:A}) $\to$ survivors; the surviving orbits are then closed by
Theorems~\ref{thm:Aprime}/\ref{thm:App} and level transport. The engine reproduces
the Shioda/da Silva anchors inside the census range ($(P_m)$ at $m=21,27$; the $m=33$ witness) and was reproduced, on all $89$
odd levels of the census list ($21\le m\le199$, $m\neq23$; the omitted small levels and $m=23$ are
classical \cite[Thms.~2.6--2.7]{daSilva}), by a second, algorithmically independent
implementation (structured-array meet-in-the-middle), and, as a third method, by a direct
brute-force enumeration --- every sorted zero-sum sextuple generated exhaustively, no
meet-in-the-middle, no code shared with either implementation --- which reproduces the
representative lists and their canonical hashes at every census level $m\le143$ (live in the
runner for $m\le45$; the shipped receipt covers the rest). After Theorem~\ref{thm:A} (which closes the
$*$-split classes at $39,117,195$), seven beyond-machinery orbits remain: four primitive
($m=33,45,105\times2$) and the induced copies at $99,135,165$; after Theorem~\ref{thm:Aprime} (which closes $45$ and both $105$, hence by Lemma~\ref{lem:infl}
also $m=135$) and Theorem~\ref{thm:App} (which closes $33$, hence $99$ and $165$),
nothing remains. Induced copies are tied to their primitives by the finite pull--push
correspondences of Lemmas~\ref{lem:infl} and~\ref{lem:descent}, which transfer claim in both
directions.

\begin{lemma}[census completeness]\label{lem:complete}
Fix $m$ and consider the shipped generator, classifiers, and canonizer.
\begin{enumerate}
\item[(i)] Every element of $M_m$ has even length $2g$, $g$ its grade: for nonzero entries
$\fr{(m-t)x}=m-\fr{tx}$, so $\sum_i\fr{(m-t)a_i}=Lm-\sum_i\fr{ta_i}$ and grade-constancy forces
$g=L/2$. In particular $M_m(1)$ consists exactly of the vanishing pairs $\{k,m-k\}$ (including the
self-pair $k=m/2$ at even $m$), and a grade-$3$ sextuple is decomposable exactly when it
contains a vanishing pair --- the complementary quadruple is then automatically grade-$2$ Hodge.
\item[(ii)] The quasi-decomposability test is complete for da Silva's definition: for a pair-free
sextuple $a$, the $8$-multiset $a\uplus\{k,m-k\}$ splits in $M_m$ only with part-lengths $(2,6)$ or
$(4,4)$, by (i); in every $(2,6)$ split the pair is either $\{k,m-k\}$ itself or $\{a_i,m-a_i\}$
with $k\in\{a_i,m-a_i\}$, and in both cases the split is the trivial one (the pair part is the
adjoined pair, the sextuple part is $a$ itself) --- excluded by the non-triviality clause of
\cite[Def.~2.4]{daSilva}. Hence $a$ is quasi-decomposable iff some
$a\uplus\{k,m-k\}$ splits into two grade-$2$ Hodge quadruples --- exactly the test performed.
\item[(iii)] The meet-in-the-middle generator produces exactly the grade-$3$ Hodge sextuples: the
join condition enforces $\sum_i\fr{ta_i}=3m$ for every $t$ in a half-unit system, and the
complement identity of (i) extends it to all units; conversely every sorted sextuple is the join of
its two sorted halves, each of which occurs in the enumeration of \emph{all} $3$-multisets of
nonzero residues.
\item[(iv)] Canonization retains the lexicographically least sorted conjugate --- a system of
distinct representatives for the $(\Z/m)^\times$-orbits.
\end{enumerate}
Consequently the census verdicts (decomposable / quasi / standard / $*$-split / survivor) coincide
with their definitions level by level; the standard classifier is complete for odd $m$ by
Proposition~\ref{prop:D}.
\end{lemma}

\begin{proof}
(i) is the two displayed identities; for the last claim, a vanishing pair contributes $m$ to
every unit sum, so the complementary quadruple has constant grade $2$. (ii): the parts have even lengths summing to $8$ with grades
$(1,3)$ or $(2,2)$; in a $(2,6)$ split the length-$2$ part is a vanishing pair inside
$a\uplus\{k,m-k\}$, and since $a$ is pair-free it is $\{k,m-k\}$ (complement $=a$) or uses one
adjoined entry, forcing $k\in\{a_i,m-a_i\}$ and complement $(a\smallsetminus\{a_i\})\uplus\{a_i\}=a$
(multiplicities included; the even-$m$ self-pair case is identical). (iii): if the joined profile
sums to $3m$ on the half-units $t$, then at $m-t$ it sums to $6m-3m=3m$; the converse containment
is the split of a sorted sextuple at position $3$. (iv): distinct orbits have distinct minima.
The implementations are additionally calibrated against direct brute force (the shipped
third-method enumeration, which reproduces representative counts and canonical hashes at every
census level $m\le143$; the even-parity engine against brute force for $m\le14$) and
reproduce the Shioda/da~Silva anchors above. The full $89$-level run ---
per-level counts, classification tallies, survivor lists, and per-orbit closure witnesses ---
ships as checksummed data artifacts with the ancillary files.
\end{proof}

\emph{What is certified, and how.} Four tiers, in decreasing logical strength. (1)~\emph{Per-orbit
certificates}: for each of the $78{,}299$ orbit representatives across the $89$ levels, a stored
witness of its classification (the vanishing pair / the quasi split / the standard parameter /
the $*$-split triples), re-verified against the definitions on every run of the supplied runner;
for the seven beyond-machinery orbits the runner re-establishes the \emph{negative} screenings
live (no vanishing pair, no quasi witness at any $k$, not standard, no zero-sum split), so the
terminal labels are re-checked certificates, not labels. (2)~The \emph{completeness lemma}
above, which reduces ``no orbit is missed'' to the definitional correctness of the shipped
generator and canonizer. (3)~The \emph{independent reimplementation}: an algorithmically
independent second program regenerates the census from scratch and matches it level by level
(its complete $89$-level log ships; the runner re-executes it live at anchor levels and
regenerates a fresh sub-range against the per-level summary table of
Appendix~\ref{app:census}). (4)~\emph{Checksummed receipts}
of the historical campaign runs: a checksum certifies file integrity, never recomputation ---
recomputation is what tiers (1)--(3) execute. A per-level summary table
(\texttt{census\_level\_summaries.json}: representative count, tally by kind, survivor list,
and the SHA-256 hash of the canonically sorted representative list) makes ``no missed orbit''
checkable level by level: any regeneration, by any implementation, must reproduce the count and
the hash.

\emph{The bound $199$.} The bound is a computational horizon, not a structural one: every
closure mechanism used here is degree-uniform, and nothing in the method stops at $199$. What
grows is the census cost --- on the reference machine (Apple M4 Max, Python 3.13) the engine
takes $158$\,s at the single level $m=199$, the full two-engine $89$-level sweep of the pinned
receipt took $\approx36$ minutes, and the complete independent rerun $\approx12$ minutes --- and $199$ was
chosen to include, with margin, the largest levels at which new primitive phenomena appear in
the odd census: the two-pair genre at $105$ (the last new primitive beyond-machinery orbit) and
the $*$-split gap family through $195=5\cdot39$. The probes beyond the bound (next paragraph)
found no new phenomenon.

Additional
single-engine censuses at the induced-risk levels beyond $199$ --- computational observations:
these three levels sit outside the witness-tier receipts of Theorem~\ref{thm:B} --- found: at
$231=33\cdot7$ and
$297=33\cdot9$ one beyond-machinery orbit each, the induced copy of the $m=33$ class (content $7$,
resp.\ $9$), closed by Theorem~\ref{thm:App} and transport; at the control level $273=3\cdot7\cdot13$,
two $*$-split classes (closed by Theorem~\ref{thm:A}) and nothing else. They indicate that the
conclusion of Theorem~\ref{thm:B} also holds at $m=231,273,297$; the theorem itself does not
depend on these levels.

\section{The field obstruction and Question~1 (Theorem~\ref{thm:C})}\label{sec:C}

\emph{Normalization.} Fix a split prime $p\equiv1\bmod m$ and a multiplicative character
$\chi\colon\F_p^\times\to\Qbar^\times$ of exact order $m$ (extended by $\chi(0)=0$). Set
\[
S_0(a)=\sum_{\substack{v\in(\F_p^\times)^6\\ v_0+\dots+v_5=0}}\ \prod_{i}\chi^{a_i}(v_i),
\qquad j(a)=\frac{S_0(a)}{p-1}\in\Z[\zt_m]
\]
(the scaling $v\mapsto\lambda v$ is free since $\sum a_i\equiv0$, so $S_0$ is exactly
divisible by $p-1$; the certificate asserts this). By Weil's count of points of diagonal
hypersurfaces \cite[pp.~500--502]{Weil49}, \cite{Weil52},
\[
\#X^4_m(\F_p)=1+p+p^2+p^3+p^4+\sum_a j(a),
\]
the sum over all $a=(a_0,\dots,a_5)$ with $a_i\neq0$ and $\sum a_i\equiv0$; the sign in front
of the character sum is $+1$ in our even dimension, which is what fixes the normalization below.
(The identity is verified exactly in the ancillary at $m=3$, $p=7$, where the left side is
computed by brute force: $2801+889=3690$.) Consequently \emph{geometric} Frobenius at $p$
acts on the \'etale eigenspace $\V(a)\subset H^4_{\mathrm{prim}}$ by $j(a)$, with
$|j(a)|=p^2$ in every embedding --- hence on $\V(a)(2)$ by $u(a,p)=j(a)/p^2$ (arithmetic Frobenius acts by the inverse scalar;
every statement in this paper is about geometric Frobenius); the conjugate eigenlines carry
$u(ta,p)$, computed by replacing $\chi$ with $\chi^t$. Under an inhomogeneous convention
($\sum v_i=1$) the Jacobi symbol changes by a sign that the point-count formula compensates;
the certificate's asserted $(u')^3=-1$ pins the geometric order regardless. A $30$-digit
numeric cross-check against the Gauss-sum product $\prod_ig(\chi^{a_i})/p^3$ corroborates
the normalization (ancillary; corroboration only --- the certificate itself is exact).
The certificate's character is pinned: $\chi$ is the character of $\F_{67}^\times$ with
$\chi(g)=\zt_{33}$ on the smallest primitive root $g=2$ (the shipped code computes discrete
logarithms base $2$); replacing $\chi$ by a conjugate character permutes the twenty conjugate
values and can exchange $\zt_6\leftrightarrow\zt_6^{-1}$ at $t=1$ --- the
convention-independent assertions are the exact order six and the two-value distribution.

\begin{proof}
Fix a prime $\ell\neq67$ and an embedding $\iota\colon\Qbar\hookrightarrow\Qbar_\ell$;
cohomology in this proof is $\ell$-adic \'etale cohomology with $\Qbar_\ell$-coefficients
(the eigenspaces correspond under comparison), and the algebraic integers
$u(ta_0,67)\in\Z[\zt_{33}]$ are read in $\Qbar_\ell$ through $\iota$; the asserted
identities $(u')^3=-1$, $(u')^6=1$, $u'\neq1$ hold in $\Z[\zt_{33}]$, hence under every
embedding, so no choice matters. The $\ell$-adic class of a cycle over $K$ is $G_K$-fixed. Take $\mathfrak p\mid67$ in $\Q(\zt_{33})$
($67\equiv1\bmod33$ splits, good reduction). $\mu_{33}\subset\F_{67}$ makes the $G$-action
$\F_{67}$-rational, so Frobenius preserves each $\V(ta_0)$ and acts on $\V(ta_0)(2)$ by
$u(ta_0,67)$ as fixed in the Normalization above. An exact
integer dynamic program in $\Z[x]/\Phi_{33}$ gives $u(a_0,67)=1+\zt_3=\zt_6$, and the $20$ unit-indexed
values $u(ta_0,67)$, $t\in(\Z/33)^\times$ --- the Frobenius scalars on the \emph{twenty} conjugate
eigenlines of $M(a_0)$: the ordered stabilizer is trivial ($7t\equiv7$ forces $t=1$), the multiset
stabilizer $\langle4\rangle$ has order $5$, so the twenty eigenlines fall into four
coordinate-permutation classes of five; the scalar --- a symmetric function of the entries --- is
constant on each class, and the four class-values coincide in pairs --- take exactly the two
values $\zt_6^{\pm1}$, each on ten eigenlines. The certificate asserts \emph{exact order six}
on every conjugate --- $(u')^3=-1$, $(u')^6=1$, $u'\neq1$ --- and the geometric scalar is
independent of the Jacobi-sum sign convention, so each eigenvalue is a primitive sixth root of
unity, hence $\neq1$. A Frobenius-fixed vector scaled by $\neq1$ vanishes componentwise.
Finally, the choice of $\mathfrak p\mid67$ is immaterial: the primes above $67$ in
$\Q(\zt_{33})$ form one $\Gal(\Q(\zt_{33})/\Q)$-orbit, and replacing $\mathfrak p$ conjugates
the chosen character $\chi$, hence permutes the twenty computed values $u(ta_0,67)$ among
themselves; since \emph{each} of them has exact order $6$, the conclusion is the same at every
prime above $67$.
\end{proof}
$W$ is defined over $\Q$ and its translates are $\Q(\zt_{33})$-rational; by the census the $m=33$ class
is the unique non-quasi non-standard class at that degree, so its projection vanishes and Question~1 is
answered negatively. The YES-branch would have proved the Hodge conjecture at these classes; the NO says
only that this candidate cannot certify it, and imposes a sharp local residue-degree constraint
on any number field of definition of a certifying cycle.

\section{The closed classes' dossiers}\label{sec:dossier}
\emph{The $m=33$ witness (closed by Theorem~\ref{thm:App})}: indicated eigenvalue field $\Q(\sqrt{-3})$ ---
throughout, \emph{eigenvalue field} means the field generated by the computed values of
$\chi_a$ at the sampled split primes, a numerical label and not the fixed field of the
multiset stabilizer --- all
entries $\equiv1\bmod3$; its Tate-normalized Jacobi-sum Hecke character $\chi_a$ is of
finite order (see the Notation; cf.\ \cite{Weil52}), and its computed values at split primes
lie in $\mu_6$: kernel primes occur
($p=859$: every conjugate value exactly $1$), and order $6$ is attained ($p=67$, the prime
Theorem~\ref{thm:C} uses; only the certified per-prime obstruction is used, and no global order
is claimed). The dossier's remaining role is the acceptance test for explicit cycles
(Theorem~\ref{thm:App}'s pass by construction). A Weil-class alternative would need a precise
carrier: Weil classes on abelian \emph{fourfolds} of Weil type are now unconditionally
algebraic (Markman \cite{Markman1}, with Schoen's degeneration argument \cite{Schoen}; for
discriminant one there are now further proofs, \cite{FloccariFu} and \cite{vGR}), but
the Shioda--Katsura realization carries $\V(w)$ on CM abelian varieties of dimension $\ge10$
(forced by the faithful $\Q(\zt_{33})$-action), and no algebraic compression onto a
$\Q(\sqrt{-3})$-Weil fourfold is known --- the route stays conditional on that open
compression, and is now unnecessary.

\emph{The closed classes $m=45,105$}: their indicated eigenvalue fields $\Q(\zt_9),\Q(\zt_7),\Q(\zt_{21})$
(degrees $6,6,12$) exceed the imaginary-quadratic multiplication field required by the published
algebraicity theorems for Weil-type classes known to us (Schoen \cite{Schoen}; Markman
\cite{Markman1}); the CM-field secant-sheaf program \cite{Markman2}, which would reach larger fields,
is explicitly conditional on a semiregularity ``not addressed yet'' (loc.\ cit.); on
the canonical CM carriers (dimensions $12$ and $24$) restriction of scalars to an imaginary quadratic
subfield shifts the cohomological degree away from $4$, so no descent to the covered case exists.
For the $\Q(\zt_{21})$ class at $105$, Theorem~\ref{thm:Aprime}'s Fermat-internal route is,
among the mechanisms compared here, at present
the only proof of algebraicity known to the author (a gap class: $\nu\notin S_{105}$,
$2\nu\in S_{105}$); the
$m=45$ and $\Q(\zt_7)$ $m=105$ classes are also reachable by the standard calculus --- $\nu\in S_{45}$,
resp.\ $\nu\in S_{105}$, by explicit four-generator $\pm1$ certificates (re-derivable via the ancillary
lattice test), hence algebraic by Lemma~4.1(iii) of \cite{Aoki00} as well, Theorem~\ref{thm:Aprime}
supplying the explicit two-pair presentation. The computed per-prime values (orders $9/42/7$ at
sampled split primes; computational observations) indicate analogous local residue-degree
constraints at the closed levels; explicit equations remain open.

\section{Methods, code availability, and provenance}\label{sec:methods}
Exact objects throughout (Gauss/Jacobi sums as algebraic integers; the decisive certificate in
$\Z[x]/\Phi_{33}$ with no floating point; the exact-twist certificate framework follows the surface
companion \cite{P1}); the only floating-point computation anywhere in the paper is the $30$-digit
Gauss-sum cross-check of \S\ref{sec:C}, which corroborates a normalization and certifies
nothing.
Every proof-critical computational claim of Theorem~\ref{thm:B} carries either a stored
per-orbit witness certificate or an algorithmically independent reimplementation; the runner
re-executes the anchor levels live and re-verifies every stored witness and every supplied
certificate against the definitions.
The census engines, the closure oracles, the lattice membership test, the exact Frobenius certificate,
and a standalone verifier of every closure identity in this paper are provided as ancillary files
(\texttt{anc/}; \texttt{run\_all.sh} reproduces the smoke tier --- identities, census anchors, both
lattice verdicts, the Theorem~\ref{thm:C} certificate, the engine calibrations, the checksum
manifest, the re-verification of all stored per-orbit witnesses (every stored class re-derived to
be a grade-$3$ Hodge character, the per-level counts and tallies asserted, the survivor negatives
re-screened), the independent modular-kernel lattice witnesses, the independent census at the
anchor levels and at the even key levels, the level-summary cross-check on a fresh sub-range, the
direct brute-force generator on every census level through $m=45$, the even-sector tier and
reader tables, and a second independent implementation of the Theorem~\ref{thm:C} certificate ---
in under a minute on the reference machine whose measured runtimes the ancillary README
reports;
\texttt{run\_full\_census.sh} is the archival tier, the
complete from-scratch independent regeneration of all $89$ levels against the pinned summary
hashes; the README records the reference configuration and the measured runtimes quoted in
\S\ref{sec:B}), together with the deep-closure ledger and the stored degree-$210$ certificate as data
artifacts; an algorithmically independent census reimplementation is bundled --- its complete $89$-level
run ships as a checksummed receipt, and the runner re-executes it live at the anchor levels ---
and the original implementation used during the verification campaign is methodologically
independent of both --- a historical instrument, \emph{not} part of the shipped reproducible
chain; no claim in this paper rests on it. \emph{Data and code availability.} The complete package --- engines, verifiers, certificates,
receipts, checksum manifest, pinned dependencies, and a single-command runner --- is submitted
as ancillary files with this manuscript and is licensed for reuse (see \texttt{LICENSE} there);
at acceptance it will additionally be deposited in a public archive under a permanent
identifier (DOI) as an immutable release, with the checksums of this submission. This
manuscript was prepared with AI assistance; all computational claims are backed by the
ancillary code and receipts, and the author takes full responsibility for all proofs, code, and
bibliographic claims.

\appendix
\section{The cited statements of da Silva and Aoki}\label{app:aoki}

For self-containedness we reproduce the exact statements used, in the numbering and notation of
the primary texts.

\emph{From \cite{daSilva}} (his $U^r_m$ is the set of characters of $X^r_m$, $C^n_m$ the
subset whose eigenspaces are spanned by algebraic cycle classes, $B^n_m$ the Hodge characters;
$\langle\;\rangle$ and the semigroup $M_m$ are as in our Notation):

\begin{itemize}
\item[--] \textbf{Theorem 2.2.} \emph{Let $n=r+s$ with $r,s\ge1$. Then there is an
isomorphism}
\[
f\colon\bigl[H^{r}_{\mathrm{prim}}(X^r_m,\C)\otimes H^{s}_{\mathrm{prim}}(X^s_m,\C)
\bigr]^{\mu_m}\ \oplus\ H^{r-1}_{\mathrm{prim}}(X^{r-1}_m,\C)\otimes
H^{s-1}_{\mathrm{prim}}(X^{s-1}_m,\C)\ \xrightarrow{\ \sim\ }\ H^n_{\mathrm{prim}}(X^n_m,\C)
\]
\emph{with the following properties: {\rm(a)} $f$ is $G^n_m$-equivariant; {\rm(b)} $f$ is a
morphism of Hodge structures of type $(0,0)$ on the first summand and of type $(1,1)$ on the
second; {\rm(c)} if $n=2p$ then $f$ preserves algebraic cycles, moreover if
$Z_1\otimes Z_2\in H^{r-1}_{\mathrm{prim}}(X^{r-1}_m,\C)\otimes
H^{s-1}_{\mathrm{prim}}(X^{s-1}_m,\C)$ then $f(Z_1\otimes Z_2)=m\,Z_1\wedge Z_2$, where
$Z_1\wedge Z_2$ is the algebraic cycle obtained by joining $Z_1$ and $Z_2$ by lines on
$X^n_m$.} --- Lemma~\ref{lem:transport} uses the \emph{second} summand at $r=s=2$, $n=4$:
$H^1_{\mathrm{prim}}(X^1_m)^{\otimes2}\to H^4_{\mathrm{prim}}(X^4_m)$, of type $(1,1)$ by (b),
carrying algebraic classes to algebraic classes by (c). No $\mu_m$-invariance is imposed on that
summand, and no hypothesis on the characters enters.
\item[--] \textbf{Notation before Cor.~2.3.} \emph{$U^{r,s}_m=\{(\beta,\gamma)\in
U^r_m\times U^s_m\mid\beta=(b_0,\dots,b_{r+1}),\ \gamma=(c_0,\dots,c_{s+1}),\
b_{r+1}+c_{s+1}=0\}$, and for $(\beta,\gamma)\in U^{r,s}_m$,
$\beta\#\gamma=(b_0,\dots,b_r,c_0,\dots,c_s)$.}
\item[--] \textbf{Corollary 2.3.} \emph{Suppose $n=2p=r+s$, where $r,s\ge1$.
{\rm(a)} If $r,s$ are odd and $(\beta',\gamma')\in C^{r-1}_m\times C^{s-1}_m$, then
$\beta'*\gamma'\in C^n_m$. {\rm(b)} If $r,s$ are even and
$(\beta,\gamma)\in(C^r_m\times C^s_m)\cap U^{r,s}_m$, then $\beta\#\gamma\in C^n_m$.}
--- the full hypothesis list of (b): $r,s$ even, both factors of known claim, and the matching
condition $b_{r+1}+c_{s+1}=0$. Pairwise distinctness of entries is \emph{not} among them,
which is what Remark~\ref{rem:selfpair} uses.
\item[--] \textbf{Definition 2.4.} \emph{An element $a\in M_m$ is called decomposable if
$a=c+d$ for some $c,d\in M_m$, otherwise indecomposable. An element is called
quasi-decomposable if $a+b=c+d$ for some $b\in M_m(1)$ and $c,d\in M_m$ with $c,d\neq a$.}
($M_m$ is his semigroup of \emph{non-negative} solutions, as recalled in the Notation; the
clause $c,d\neq a$ is the non-triviality used in Lemma~\ref{lem:complete}(ii).)
\end{itemize} \emph{Notation translation:} Aoki's $\mathfrak B^n_m$ is our $B^n_m$ (Hodge characters
of $X^n_m$); his $\mathfrak D^n_m$ is our $\Ds^n_m$; his \textsc{claim} is our
$\mathrm{claim}$; his $*$ is our juxtaposition; his multiplicity vector $u(\alpha)$ is our
$\nu(\alpha)$ (we reserve $u$ for the local Frobenius scalar); his $S_m$ and gap group are the
objects of Proposition~\ref{prop:ident}. Indices: he writes $n$ for the cohomological degree,
so his $X^n_m$ with $n=4$ is our fourfold and $n=2$ its surface.

\emph{From \cite{Aoki87}} (notation as there: $\mathfrak B^n_m$ the Hodge characters of $X^n_m$,
$\mathfrak D^n_m$ the classes that are coordinate permutations of $(a_0,-a_0,\dots,a_r,-a_r)$,
$r=n/2$; $*$ juxtaposition; \textsc{claim}$(\alpha)$: $\V(\alpha)$ is generated by classes of
algebraic cycles on $X^n_m$):

\begin{itemize}
\item[--] \textbf{Theorem 1-1} (Shioda). \emph{The linear space $L$ represents
$\delta\in\mathfrak D^n_m$; more precisely
$\omega_\delta(L)\cdot\overline{\omega_\delta(L)}=(-1)^r m^{n+1}$.}
\item[--] \textbf{Theorem 1-4} (Shioda \cite{Shioda79}, Ran \cite{Ran}). \emph{Let $r$ and $s$ be non-negative even integers
such that $n=r+s+2$, and let $\alpha\in\mathfrak B^r_m$, $\beta\in\mathfrak B^s_m$. Then:
(i) if \textup{\textsc{claim}}$(\alpha)$ and \textup{\textsc{claim}}$(\beta)$ are true, then
\textup{\textsc{claim}}$(\alpha*\beta)$ is also true; (ii) if there exists $\delta\in\mathfrak D^s_m$ such
that \textup{\textsc{claim}}$(\alpha*\delta)$ is true, then \textup{\textsc{claim}}$(\alpha)$ is also true.}
\item[--] \textbf{Theorem 2-1.} \emph{For an odd prime $p\mid m$, $d=m/p$, $\gcd(a,d)=1$, the
variety $Y$ defined by (2.1) there is a subvariety of $X^{p-1}_m$ of codimension $r=(p-1)/2$
representing $\alpha=\sigma_{p,a}$; more precisely
$\omega_\alpha(Y)\cdot\overline{\omega_\alpha(Y)}=(-1)^r p^{p-2}m^p$.}
\item[--] \textbf{\S1 definition.} \emph{For a prime $p\mid m$, $d=m/p$, and each $i$ with
$d/(i,d)>2$: $\sigma_{p,i}=(i,\,i+d,\dots,i+(p-1)d,\;m-pi)$ for $p\ge3$ --- the ($p$-)standard
elements.}
\end{itemize}

\emph{From \cite{Aoki00}} (verified verbatim against the open-access copy in the Rikkyo
University repository, doi:10.14992/00009819). Two errata of the printed text, both verified
against the primary and neither affecting the statements as used here: the companion CMUSP
\textbf{51} (2002) paper's attribution line after its Theorem~1.2 assigns part (ii)
(= Theorem~0.1(i) below) to the Catalan-curves reference; and the $m=28$ gap-generator entry
$\xi_{28}$ on p.~185 is a misprint (not a zero-sum character) --- a machine-verified
\emph{possible replacement} of exactly the join form the text requires
($\xi_m\in B^2\cup(B^4\cap(\mathfrak A^1*\mathfrak A^1))$) is
$\xi_{28}=(1,9,18)*(10,21,25)$, with $\nu\notin S_{28}$, $2\nu\in S_{28}$ (whether it is the
intended generator modulo $S_{28}$ is not established; the $168$ closure of \S\ref{sec:even}
does not depend on it):

\begin{itemize}
\item[--] \textbf{Theorem 0.1.} \emph{Suppose the prime factorization of $m$ is one of the
following forms: (i) $m=2^a3^b5^c7^d$, where $a,b,c,d$ are non-negative integers such that either
$c=0$ or $d=0$; (ii) $m=p^e$ or $2p^e$, $p$ an odd prime. Then the Hodge conjecture is true for
all abelian varieties of Fermat type of degree $m$.} Its proof (\S4 there) opens: ``In view of
Corollary~3.2, it suffices to show that the Hodge conjecture for $X^n_m$ is true for all $n$'' ---
the bridge to the Fermat varieties themselves used in \S\ref{sec:even}; the
verification then reduces, via Lemma~4.1 and Theorem~2.1 there, to the gap generators $\xi_m$ at
$m=12,15,20,21,28$ (the $m=28$ entry being the misprint recorded above).
\item[--] \textbf{Lemma 4.1} (verbatim). \emph{Let $\alpha,\beta\in B_m$. Then:
(i) if both $\V(\alpha)$ and $\V(\beta)$ are algebraic, then so is $\V(\alpha*\beta)$;
(ii) if $\V(\alpha*\delta)$ is algebraic for some $\delta\in D_m$, then so is $\V(\alpha)$;
(iii) if $\alpha\in S_m$, then $\V(\alpha)$ is algebraic.} (Here $B_m=\bigcup_nB^n_m$ and
$\V(\alpha)\subset H^n_{\mathrm{prim}}(X^n_m)$ for the $n$ with $\#\alpha=n+2$; we use it at
$n=4$.) No hypothesis on $m$ is imposed; the
groups $S_m$ and the ``gap group'' are Aoki's own objects (\S2 there), and
Proposition~\ref{prop:ident} identifies the implemented lattice with $S_m$
generator-by-generator (his \S2 also proves, via Theorem~2.1, that $2\alpha\in S_m$ for every
$\alpha\in B_m$ --- the doubling our $2u$-certificates instantiate).
\end{itemize}

\section{\texorpdfstring{The implemented lattice is Aoki's $S_m$}{The implemented lattice is Aoki's S\_m}}\label{app:ident}

Proved here because the odd-degree novelty labels of \S\ref{sec:intro} rest on it (\S\ref{sec:even}).

\begin{proposition}[the implemented lattice is Aoki's $S_m$]\label{prop:ident}
Aoki \cite[\S2]{Aoki00} defines $A_m$ as the zero-sum sublattice of the free abelian group on
$\Z/m\smallsetminus\{0\}$; $u$ --- our $\nu$ --- as the multiplicity-vector map, additive
under juxtaposition, $\nu(\alpha*\beta)=\nu(\alpha)+\nu(\beta)$; the standard elements,
verbatim, as
$\sigma_{p,a}=(a,\,a+\frac mp,\dots,a+\frac{(p-1)m}p,\,m-pa)$ for $p\ge3$ and
$\sigma_{2,a}=(a,\,a+\frac m2,\,m-2a,\,\frac m2)$ for $p=2$, over primes $p\mid m$ ($m>p$) and
$0<a<\frac mp$ with nonzero entries; $\mathfrak S_m$ as the set of $\alpha$ admitting
$\alpha*\delta\sim\sigma_1*\cdots*\sigma_k*\delta'$ with $\delta,\delta'$ unions of vanishing
pairs, with $\mathfrak D_m\subset\mathfrak S_m$; and $S_m\subset A_m$ as the subgroup generated
by $\nu(\mathfrak S_m)$. Consequently $S_m$ equals the $\Z$-span of the pair vectors
$\nu(\{k,m-k\})$ and the standard vectors $\nu(\sigma_{p,a})$, $p=2$ included --- exactly the
generator list of the implemented lattice --- so the machine verdicts $\nu\in S_m$ /
$\nu\notin S_m$ are statements about Aoki's group itself.
\end{proposition}
\begin{proof}
$\supseteq$: every pair class lies in $\mathfrak D_m\subset\mathfrak S_m$, and every standard
element $\sigma$ satisfies $\sigma*\delta\sim\sigma*\delta$, so $\sigma\in\mathfrak S_m$: each
listed generator is in $\nu(\mathfrak S_m)$. $\subseteq$: if
$\alpha*\delta\sim\sigma_1*\cdots*\sigma_k*\delta'$, then
$\nu(\alpha)=\sum_i\nu(\sigma_i)+\nu(\delta')-\nu(\delta)$, an integer combination of the listed
generators. The two spans coincide.
\end{proof}

\section{The census, level by level}\label{app:census}

\phantomsection\label{tab:census}

Every odd level of the census list, taken from the shipped summary table
(\texttt{census\_level\_\allowbreak summaries.json}; each row there also carries the SHA-256 of
its canonically sorted representative list). Columns: number of Galois-orbit representatives; decomposable;
quasi-decomposable; Aoki standard; $*$-split (Theorem~\ref{thm:A}); closed by
Theorem~\ref{thm:Aprime}/\ref{thm:App} or level transport; and remaining --- zero at every
level, which is Theorem~\ref{thm:B}.

\begin{center}\scriptsize
\setlength{\tabcolsep}{3pt}
\begin{tabular}{rrrrrrrr|rrrrrrrr}
\toprule
$m$ & orb & D & QD & S & $*$ & A$'$/A$''$ & rem & $m$ & orb & D & QD & S & $*$ & A$'$/A$''$ & rem \\
\midrule
$21$ & $64$ & $61$ & $3$ & $0$ & $0$ & $0$ & $0$ & $113$ & $551$ & $551$ & $0$ & $0$ & $0$ & $0$ & $0$ \\
$25$ & $39$ & $38$ & $0$ & $1$ & $0$ & $0$ & $0$ & $115$ & $772$ & $771$ & $0$ & $1$ & $0$ & $0$ & $0$ \\
$27$ & $68$ & $67$ & $1$ & $0$ & $0$ & $0$ & $0$ & $117$ & $1117$ & $1113$ & $2$ & $0$ & $2$ & $0$ & $0$ \\
$29$ & $40$ & $40$ & $0$ & $0$ & $0$ & $0$ & $0$ & $119$ & $776$ & $776$ & $0$ & $0$ & $0$ & $0$ & $0$ \\
$31$ & $46$ & $46$ & $0$ & $0$ & $0$ & $0$ & $0$ & $121$ & $694$ & $694$ & $0$ & $0$ & $0$ & $0$ & $0$ \\
$33$ & $104$ & $102$ & $1$ & $0$ & $0$ & $1$ & $0$ & $123$ & $1091$ & $1090$ & $1$ & $0$ & $0$ & $0$ & $0$ \\
$35$ & $92$ & $91$ & $0$ & $1$ & $0$ & $0$ & $0$ & $125$ & $866$ & $864$ & $0$ & $2$ & $0$ & $0$ & $0$ \\
$37$ & $64$ & $64$ & $0$ & $0$ & $0$ & $0$ & $0$ & $127$ & $694$ & $694$ & $0$ & $0$ & $0$ & $0$ & $0$ \\
$39$ & $140$ & $137$ & $1$ & $0$ & $2$ & $0$ & $0$ & $129$ & $1195$ & $1194$ & $1$ & $0$ & $0$ & $0$ & $0$ \\
$41$ & $77$ & $77$ & $0$ & $0$ & $0$ & $0$ & $0$ & $131$ & $737$ & $737$ & $0$ & $0$ & $0$ & $0$ & $0$ \\
$43$ & $85$ & $85$ & $0$ & $0$ & $0$ & $0$ & $0$ & $133$ & $960$ & $960$ & $0$ & $0$ & $0$ & $0$ & $0$ \\
$45$ & $245$ & $237$ & $6$ & $1$ & $0$ & $1$ & $0$ & $135$ & $1767$ & $1756$ & $8$ & $2$ & $0$ & $1$ & $0$ \\
$47$ & $100$ & $100$ & $0$ & $0$ & $0$ & $0$ & $0$ & $137$ & $805$ & $805$ & $0$ & $0$ & $0$ & $0$ & $0$ \\
$49$ & $128$ & $128$ & $0$ & $0$ & $0$ & $0$ & $0$ & $139$ & $829$ & $829$ & $0$ & $0$ & $0$ & $0$ & $0$ \\
$51$ & $217$ & $216$ & $1$ & $0$ & $0$ & $0$ & $0$ & $141$ & $1414$ & $1413$ & $1$ & $0$ & $0$ & $0$ & $0$ \\
$53$ & $126$ & $126$ & $0$ & $0$ & $0$ & $0$ & $0$ & $143$ & $1062$ & $1062$ & $0$ & $0$ & $0$ & $0$ & $0$ \\
$55$ & $197$ & $196$ & $0$ & $1$ & $0$ & $0$ & $0$ & $145$ & $1204$ & $1203$ & $0$ & $1$ & $0$ & $0$ & $0$ \\
$57$ & $265$ & $264$ & $1$ & $0$ & $0$ & $0$ & $0$ & $147$ & $1816$ & $1812$ & $4$ & $0$ & $0$ & $0$ & $0$ \\
$59$ & $155$ & $155$ & $0$ & $0$ & $0$ & $0$ & $0$ & $149$ & $950$ & $950$ & $0$ & $0$ & $0$ & $0$ & $0$ \\
$61$ & $166$ & $166$ & $0$ & $0$ & $0$ & $0$ & $0$ & $151$ & $976$ & $976$ & $0$ & $0$ & $0$ & $0$ & $0$ \\
$63$ & $407$ & $400$ & $7$ & $0$ & $0$ & $0$ & $0$ & $153$ & $1816$ & $1814$ & $2$ & $0$ & $0$ & $0$ & $0$ \\
$65$ & $268$ & $267$ & $0$ & $1$ & $0$ & $0$ & $0$ & $155$ & $1370$ & $1369$ & $0$ & $1$ & $0$ & $0$ & $0$ \\
$67$ & $199$ & $199$ & $0$ & $0$ & $0$ & $0$ & $0$ & $157$ & $1054$ & $1054$ & $0$ & $0$ & $0$ & $0$ & $0$ \\
$69$ & $372$ & $371$ & $1$ & $0$ & $0$ & $0$ & $0$ & $159$ & $1780$ & $1779$ & $1$ & $0$ & $0$ & $0$ & $0$ \\
$71$ & $222$ & $222$ & $0$ & $0$ & $0$ & $0$ & $0$ & $161$ & $1379$ & $1379$ & $0$ & $0$ & $0$ & $0$ & $0$ \\
$73$ & $235$ & $235$ & $0$ & $0$ & $0$ & $0$ & $0$ & $163$ & $1135$ & $1135$ & $0$ & $0$ & $0$ & $0$ & $0$ \\
$75$ & $561$ & $556$ & $3$ & $2$ & $0$ & $0$ & $0$ & $165$ & $2743$ & $2736$ & $4$ & $2$ & $0$ & $1$ & $0$ \\
$77$ & $346$ & $346$ & $0$ & $0$ & $0$ & $0$ & $0$ & $167$ & $1190$ & $1190$ & $0$ & $0$ & $0$ & $0$ & $0$ \\
$79$ & $274$ & $274$ & $0$ & $0$ & $0$ & $0$ & $0$ & $169$ & $1322$ & $1322$ & $0$ & $0$ & $0$ & $0$ & $0$ \\
$81$ & $504$ & $502$ & $2$ & $0$ & $0$ & $0$ & $0$ & $171$ & $2236$ & $2234$ & $2$ & $0$ & $0$ & $0$ & $0$ \\
$83$ & $301$ & $301$ & $0$ & $0$ & $0$ & $0$ & $0$ & $173$ & $1276$ & $1276$ & $0$ & $0$ & $0$ & $0$ & $0$ \\
$85$ & $437$ & $436$ & $0$ & $1$ & $0$ & $0$ & $0$ & $175$ & $2015$ & $2012$ & $0$ & $3$ & $0$ & $0$ & $0$ \\
$87$ & $570$ & $569$ & $1$ & $0$ & $0$ & $0$ & $0$ & $177$ & $2187$ & $2186$ & $1$ & $0$ & $0$ & $0$ & $0$ \\
$89$ & $345$ & $345$ & $0$ & $0$ & $0$ & $0$ & $0$ & $179$ & $1365$ & $1365$ & $0$ & $0$ & $0$ & $0$ & $0$ \\
$91$ & $473$ & $473$ & $0$ & $0$ & $0$ & $0$ & $0$ & $181$ & $1396$ & $1396$ & $0$ & $0$ & $0$ & $0$ & $0$ \\
$93$ & $646$ & $645$ & $1$ & $0$ & $0$ & $0$ & $0$ & $183$ & $2334$ & $2333$ & $1$ & $0$ & $0$ & $0$ & $0$ \\
$95$ & $539$ & $538$ & $0$ & $1$ & $0$ & $0$ & $0$ & $185$ & $1930$ & $1929$ & $0$ & $1$ & $0$ & $0$ & $0$ \\
$97$ & $409$ & $409$ & $0$ & $0$ & $0$ & $0$ & $0$ & $187$ & $1761$ & $1761$ & $0$ & $0$ & $0$ & $0$ & $0$ \\
$99$ & $825$ & $822$ & $2$ & $0$ & $0$ & $1$ & $0$ & $189$ & $3100$ & $3091$ & $9$ & $0$ & $0$ & $0$ & $0$ \\
$101$ & $442$ & $442$ & $0$ & $0$ & $0$ & $0$ & $0$ & $191$ & $1552$ & $1552$ & $0$ & $0$ & $0$ & $0$ & $0$ \\
$103$ & $460$ & $460$ & $0$ & $0$ & $0$ & $0$ & $0$ & $193$ & $1585$ & $1585$ & $0$ & $0$ & $0$ & $0$ & $0$ \\
$105$ & $1286$ & $1276$ & $6$ & $2$ & $0$ & $2$ & $0$ & $195$ & $3715$ & $3707$ & $4$ & $2$ & $2$ & $0$ & $0$ \\
$107$ & $495$ & $495$ & $0$ & $0$ & $0$ & $0$ & $0$ & $197$ & $1650$ & $1650$ & $0$ & $0$ & $0$ & $0$ & $0$ \\
$109$ & $514$ & $514$ & $0$ & $0$ & $0$ & $0$ & $0$ & $199$ & $1684$ & $1684$ & $0$ & $0$ & $0$ & $0$ & $0$ \\
$111$ & $900$ & $899$ & $1$ & $0$ & $0$ & $0$ & $0$ &  & & & & & & &  \\
\bottomrule
\end{tabular}
\end{center}

\noindent Totals over the $89$ levels: $78{,}299$ orbit representatives, split as
$78{,}181$ decomposable, $79$ quasi-decomposable, $26$ standard, $6$ $*$-split and $7$
beyond-machinery --- and $0$ remaining at every level.

\section{Computational outlook: the even-degree boundary through degree 250}\label{sec:even}

\emph{This appendix is independent of Theorems~\ref{thm:A}--\ref{thm:C} and of the odd census;
it records a computed map whose systematic treatment is the companion paper \cite{BNeven}.}

The odd-degree theorems are complete. The even-degree boundary is a different landscape, and this
closing section only \emph{reports its computed map} --- engines, closure ledger, certificates,
and pinned receipts ship in the ancillary --- deferring systematic treatment (full statements and
proofs of the closure and wall structure) to the companion paper \cite{BNeven}, with which this
note shares its ancillary package. Nothing in Theorems~\ref{thm:A}--\ref{thm:C} or in the odd
census depends on this section, and every computational claim below is backed
receipt-by-receipt within the ancillary of \emph{this} submission --- the companion adds
statements and proofs, not data. All censuses use the
parity-corrected decomposability test (the self-paired element $m/2$; part~(i) of
Lemma~\ref{lem:complete} covers both parities); every lattice verdict below carries an
independent certificate in the ancillary --- a modular kernel witness for each non-membership, a
re-summed integer combination for each membership. One statement --- the identification of the implemented lattice with Aoki's $S_m$ --- is
load-bearing also for the odd-degree novelty classification above, and is therefore proved in
this paper, in Appendix~\ref{app:ident} (Proposition~\ref{prop:ident}).

\emph{The computed map} {\rm(computed; engines, closure ledger, and receipts in the ancillary;
full statements and proofs in \cite{BNeven})}. The even census finds forty primitive
orbits surviving the four base predicates \emph{and} the complete depth-$\le2$ closure oracle
through degree $250$ ($21$ at $m\le108$, the first at $m=50$ --- the odd law ``survivor
$\Rightarrow3\mid m$'' fails in the even sector --- then $14$ through $200$ and $5$ more
through $250$). The tier matters: unlike the odd census, where the base predicates alone leave
only the seven orbits of Theorem~\ref{thm:B}, the even base tier is far larger (first survivor at
$m=32$), and all of it except those forty closes at depth $\le2$; the tiers are defined
and tabulated in \cite{BNeven}. Iterated partition certificates of the
Theorem~\ref{thm:Aprime}/\ref{thm:App} kind close all but ten, and an exhaustive scan of
claim-equivalence moves organizes the residual into three walls. Two more classes then close:
$m=168$ by a coset transfer --- its $\nu$ is congruent modulo $S_{168}$ to that of an explicit
$*$-split class, which Theorem~\ref{thm:A} closes (Aoki's degree-form Theorem~0.1(i)
\cite{Aoki00} also covers $168=2^3\cdot3\cdot7$; the bridge is quoted in
Appendix~\ref{app:aoki}, together with two verified errata of the printed text, and neither is
relied upon) --- and the isolated degree-$210$
class closes by an explicit $40$-generator $\pm1$ lattice certificate $\nu\in S_{210}$ (pairs
and standard elements only, converted into a literal Aoki chain by the negation closure), whence
$\V$ is algebraic by Theorems~1-4(i),(ii) of \cite{Aoki87}. What remains \emph{unresolved by
the implemented closure family and the cited published theorems} is eight certified gap classes
--- all with $\nu\notin S_m$, $2\nu\in S_m$ (the doubling holds for every Hodge class by
\cite[Thm.~2.1]{Aoki00}), all outside the degree forms of Aoki's theorem --- in three claim-equivalence walls
$W_{70},W_{110},W_{114}$ (labelled by their lowest level; a wall spans several levels).

Qualitatively, the even sector differs by exactly the feature that made Theorem~\ref{thm:App}
work: the self-paired element $m/2$ exists only at even levels
(Remark~\ref{rem:selfpair}), so quasi-decomposition has a channel there that no odd level
provides --- which is why the odd boundary closes and the even one, at present, does not. The
wall table, the per-class dossiers, the deep negative sweeps, and the independent even-parity
verification are in \cite{BNeven}.


\begin{thebibliography}{99}
\raggedright
\bibitem{Aoki87} N.~Aoki, \emph{Some new algebraic cycles on Fermat varieties}, J.~Math.\ Soc.\ Japan
\textbf{39} (1987), no.~3, 385--396; doi:10.2969/jmsj/03930385.
\bibitem{Aoki00} N.~Aoki, \emph{Some remarks on the Hodge conjecture for abelian varieties of Fermat
type}, Comment.\ Math.\ Univ.\ St.\ Paul.\ \textbf{49} (2000), no.~2, 177--194;
doi:10.14992/00009819 (open access, Rikkyo University repository).
\bibitem{daSilva} G.~da Silva Jr., \emph{Notes on the Hodge conjecture for Fermat varieties},
Experimental Results \textbf{2} (2021), e22, doi:10.1017/exp.2021.14; arXiv:2101.04739.
\bibitem{P1} R.~Jumagulov, \emph{Galois-invariant N\'eron--Severi ranks of Fermat surfaces over
number fields: a Galois module, closed forms, a threshold, and exact tables}, arXiv:2607.17387.
\bibitem{BNeven} R.~Jumagulov, \emph{Residual gap classes in the even-degree Fermat-fourfold
census through degree 250: exchange walls and two closures}, companion paper, 2026 (submitted
alongside; the ancillary package is shared).
\bibitem{FloccariFu} S.~Floccari, L.~Fu, \emph{The Hodge conjecture for Weil fourfolds with
discriminant $1$ via singular OG6-varieties}, J.~Math.\ Pures Appl.\ \textbf{210} (2026),
103876; doi:10.1016/j.matpur.2026.103876.
\bibitem{Markman1} E.~Markman, \emph{Cycles on abelian $2n$-folds of Weil type from secant sheaves
on abelian $n$-folds}, arXiv:2502.03415.
\bibitem{Markman2} E.~Markman, \emph{Secant sheaves on abelian $n$-folds with real multiplication
and Weil classes on abelian $2n$-folds with complex multiplication}, arXiv:2509.23079.
\bibitem{Ran} Z.~Ran, \emph{Cycles on Fermat hypersurfaces}, Compositio Math.\ \textbf{42}
(1980/81), no.~1, 121--142.
\bibitem{Schoen} C.~Schoen, \emph{Hodge classes on self-products of a variety with an automorphism},
Compositio Math.\ \textbf{65} (1988), no.~1, 3--32; Addendum, \textbf{114} (1998), no.~3,
329--336; doi:10.1023/A:1000566205021.
\bibitem{Shioda79} T.~Shioda, \emph{The Hodge conjecture for Fermat varieties}, Math.\ Ann.\ \textbf{245}
(1979), no.~2, 175--184; doi:10.1007/BF01428804.
\bibitem{SK} T.~Shioda, T.~Katsura, \emph{On Fermat varieties}, T\^ohoku Math.\ J.\ (2)
\textbf{31} (1979), no.~1, 97--115; doi:10.2748/tmj/1178229881.
\bibitem{vGR} B.~van Geemen, A.~Rapagnetta, \emph{Hyperk\"ahler sixfolds, abelian fourfolds of
Weil type and a Hodge class}, arXiv:2607.18341.
\bibitem{Weil49} A.~Weil, \emph{Numbers of solutions of equations in finite fields}, Bull.\ Amer.\
Math.\ Soc.\ \textbf{55} (1949), 497--508; doi:10.1090/S0002-9904-1949-09219-4.
\bibitem{Weil52} A.~Weil, \emph{Jacobi sums as ``Gr\"ossencharaktere''}, Trans.\ Amer.\ Math.\ Soc.\
\textbf{73} (1952), 487--495; doi:10.2307/1990804.
\end{thebibliography}
\end{document}